\documentclass[reqno,12pt]{amsart} 
\usepackage{amsmath,amssymb,amsthm,enumerate,mathrsfs,colonequals}
\usepackage{fullpage,url,tikz,xspace,setspace}
\usepackage{microtype}
\usepackage{calc}

\usepackage[
]{hyperref}
\usepackage[alphabetic,lite,nobysame]{amsrefs} 

\def\be{\begin{equation}}
\def\ee{\end{equation}}

\def\0{^{\phantom 0}}
\def\9{_{\phantom q}}

\DeclareMathOperator\val{val}

\font\koppa=grmn1200

\newtheorem{theorem}{Theorem}[section]
\newtheorem{lemma}[theorem]{Lemma}
\newtheorem{proposition}[theorem]{Proposition}
\theoremstyle{remark}
\newtheorem{remark1}[theorem]{Remark}
\theoremstyle{definition}
\newtheorem{definition1}[theorem]{Definition}

\newtheorem{riemannhypothesis}[theorem]{Riemann Hypothesis}

\newcommand{\Jacobi}[2]{\genfrac{(}{)}{}{1}{#1}{#2}}

\newcommand{\Q}{\mathbf Q}
\newcommand{\Z}{\mathbb Z}

\newcommand{\fb}{\mathfrak b}
\newcommand{\fq}{\mathfrak q}
\newcommand{\OO}{\mathcal O}
\newcommand{\RR}{\mathcal R}

\DeclareMathOperator\SL{SL}
\DeclareMathOperator\SLT{\SL_{2}}
\begin{document}

\title{On the bounded generation of arithmetic \protect{\boldmath{$\SLT$}}}


\author{Bruce~W.~Jordan}
\address{Department of Mathematics, Baruch College, The City University
of New York, One Bernard Baruch Way, New York, NY 10010-5526, USA}
\email{bruce.jordan@baruch.cuny.edu}

\author{Yevgeny~Zaytman}
\address{Center for Communications Research, 805 Bunn Drive, Princeton, 
NJ 08540-1966, USA}
\email{ykzaytm@idaccr.org}

\begin{abstract}
Let $K$ be a number field and ${\mathcal O}$ be the ring of
$S$-integers in $K$.  Morgan, Rapinchuck, and Sury have proved that if
the group of units ${\mathcal O}^{\times}$ is infinite, then every
matrix in ${\rm SL}_2({\mathcal O})$ is a product of at most $9$
elementary matrices.  We prove that under the additional hypothesis
that $K$ has at least one real embedding or $S$ contains a finite
place we can get a product of at most $8$ elementary matrices.  If we
assume a suitable Generalized Riemann Hypothesis, then every matrix in
${\rm SL}_2({\mathcal O})$ is the product of at most $5$ elementary
matrices if $K$ has at least one real embedding, the product of at most $6$
elementary matrices if $S$ contains a finite place, and the product of at
most $7$ elementary matrices in general.
\end{abstract}

\maketitle

\section{Introduction}
\label{rabbit}

Let $K$ be a number field and $S$ be a finite set
of primes of $K$ containing the archimedean valuations.
Denote by $\OO=\OO_S$ the ring of $S$-integers in $K$:
\[
\OO=\OO_{S}=\{x\in K^{\times} \mid v(x)\geq 0 \text{ for all $v\notin S$}\}.
\]
For $x\in \OO$ we define the upper triangular matrix $U(x)$
and the lower triangular matrix $L(x)$ by
\begin{equation}
  \label{goose}
  U(x):= \begin{bmatrix} 1 & x\\0& 1\end{bmatrix}\quad 
    \mbox{and}\quad L(x):=\begin{bmatrix} 1 & 0\\
    x & 1\end{bmatrix} .
\end{equation}
The {\em elementary matrices} over $\OO$ are the
matrices $U(x)$, $L(x)$ for $x\in\OO$.

Consider the case where $K$ is the field of rational
numbers $\Q$.  Taking $\OO=\Z$ we have that 
every $A\in\SLT(\Z)$ is a product of elementary
matrices, but the number required is unbounded.
However, if we take $\OO=\Z[1/p]$ for $p$ prime, the
situation is different.  Every matrix $A\in\SLT(\Z[1/p])$
is a product of at at most $5$ elementary matrices as was
proved by  Vsemirnov \cite[Theorem 1.1]{v}.

The key difference between the $\Z$ and $\Z[1/p]$ for this
bounded generation question for $\SLT$ is their units:
$\Z^\times=\langle\pm 1\rangle$ is finite whereas $\Z[1/p]^\times$
is infinite.
Vaser{\v{s}}te{\u\i}n \cite{vas} proved that if $\OO$ has infinitely many
units, then $\SLT(\OO)$ is generated by elementary matrices.
Morgan, Rapinchuk, and Sury \cite[Theorem~1.1]{mrs}
recently proved an explicit general result on
bounded generation:

\begin{theorem}[Morgan, Rapinchuk, and Sury]
  \label{pint}
    Assume that the group of units $\OO^{\times}$ is infinite.
  Then every matrix in $\SLT(\OO)$ can be written as a product of at most $9$
  elementary matrices with the first one lower triangular.
\end{theorem}
\noindent The lower triangular assertion follows from their
proof: see \cite[Eq. (21) and following]{mrs}.

Here we prove two theorems on a
matrix $A\in\SLT(\OO)$:

\begin{theorem}
  \label{pinky}
  Suppose that $S$ contains a finite place or suppose that the group of units
  $\OO^{\times}$ is infinite and $K$ has at least one real embedding. 
  Then 
  $A\in\SLT(\OO)$ can be written as the product of at most $8$ 
elementary matrices with the first one
lower triangular.
  \end{theorem}

\begin{theorem}
  \label{pink}
  Assume that the group of units $\OO^{\times}$ is infinite and assume
  the Generalized Riemann Hypothesis~{\rm \ref{GRH}}.  Then
  $A\in\SLT(\OO)$ can be written as the product of at most $5$
  elementary matrices if $K$ has at least one real embedding, the
  product of at most $6$ elementary matrices if $S$ contains a finite
  place, and the product of at most $7$ elementary matrices in
  general with the first one lower triangular in each case.
  \end{theorem}

We give diophantine applications of Theorems \ref{pinky} and
\ref{pink} in \cite{jz}.  These applications require us to know that
the first matrix in our factorization into elementary matrices can be
taken to be lower triangular.  Hence we keep track of this here,
whereas it is not a concern in \cite{mrs}.

\section{Theorem \texorpdfstring{\ref{pinky}}{1.2}}
\label{punted}
\subsection{Reducing the first row of a matrix 
\texorpdfstring{\protect{\boldmath{$A\in\SLT(\OO)$}}}{A in SL_2(O)}}
\label{bronze}

Following \cite[Section 4]{mrs}, let
\begin{equation}
\label{shrimp}
\RR(\OO)=\{(a,b)\in \OO^2\mid a\OO+b\OO=\OO\}.
\end{equation}
The $(a,b)\in \RR(\OO)$ are precisely the first rows of matrices in
$\SLT(\OO)$. The effect on the first row of a matrix 
$A=\left[\begin{smallmatrix}a & b\\ * & *
\end{smallmatrix}\right]\in\SLT(\OO)$ from right multiplying by an
elementary matrix as in \eqref{goose} is
\begin{eqnarray}
\label{duck}
AL(x)& =&  \begin{bmatrix} a & b\\ * & * \end{bmatrix}\begin{bmatrix} 
1 & 0\\x& 1\end{bmatrix}=\begin{bmatrix} a+bx & b\\ * & *\end{bmatrix},\\
\nonumber AU(x)& =&  \begin{bmatrix} a & b\\ * & * \end{bmatrix}\begin{bmatrix} 
1 & x\\ 0 & 1\end{bmatrix}=\begin{bmatrix} a & ax+b \\ * & *\end{bmatrix}
\end{eqnarray}
for $x\in\OO$.

The following succinct notation using only the first rows of matrices 
is convenient:   
\begin{definition1}
\label{dime}
{\rm
For $x\in\OO$ and $(a,b)\in\RR(\OO)$, set $(a,b)\ell(x)=
(a+bx,b)$ and $(a,b)u(x)=(a, ax+b)$.

If there exist $x_1, \ldots , x_k\in\OO$ with
\begin{equation}
(c,d)=\begin{cases} (a,b)\ell(x_1)u(x_2)\cdots \ell(x_k)\quad\text{$k$ odd}\\
(a,b)\ell(x_1)u(x_2)\cdots u(x_k)\quad\text{$k$ even}
\end{cases}
\end{equation}
for $(a,b), (c,d)\in\RR(\OO)$, write $(a,b)\overset{k, \ell}{\Longrightarrow}
(c,d)$.
Similarly, if 
there exist $x_1, \ldots , x_k\in\OO$ with
\begin{equation}
(c,d)=\begin{cases} (a,b)u(x_1)\ell(x_2)\cdots u(x_k)\quad\text{$k$ odd}\\
(a,b)u(x_1)\ell(x_2)\cdots \ell(x_k)\quad\text{$k$ even}
\end{cases}
\end{equation}
for $(a,b), (c,d)\in\RR(\OO)$, write $(a,b)\overset{k, u}{\Longrightarrow}
(c,d)$.
As in \cite[Section 4]{mrs}, write $(a,b)\overset{k}{\Longrightarrow}
(c,d)$ if $(a,b)\overset{k,\ell}{\Longrightarrow} (c,d)$ or
$(a,b)\overset{k,u}{\Longrightarrow} (c,d)$.
}
\end{definition1}

\subsection{The Proof of Theorem \texorpdfstring{\ref{pinky}}{1.2}}
  \label{iron}

First we need the following Lemma \ref{lemma}, which requires
a definition.
\begin{definition1}
  \label{opiate}
  {\rm \cite[Section~3.1]{mrs}.} A prime $\fq$ of the number field
  $K$ lying above the rational prime $q$ is \emph{$\mathbb{Q}$-split}
  if $q>2$ and $K_\fq\cong \mathbb{Q}_q$.
  \end{definition1}

If $K$ is a number field, $\mu=\mu(K)$ is the number of roots of unity
in $K$.
\begin{lemma}
\label{lemma}
{\rm (cf. \cite[Lemma 4.4]{mrs}.)} Suppose $K$ has at least one real embedding or $S$ contains a finite place
and $(a,b)\in\RR(\OO)$.  Let $\mu=\mu(K)$. Then there exists 
$a'\in\OO$ and infinitely many $\mathbb{Q}$-split prime principal ideals $\fq$
of $\OO$ with a generator {\koppa\char19}
such that for any $m\equiv 1
\bmod{\phi(a'\OO)}$ we have
$(a,b)\overset{3,u}{\Longrightarrow}(a',\mbox{{\koppa\char19}}^{\mu m})$.
\end{lemma}
\begin{proof}
Let $v$ be either a real place of $K$ or a finite place in $S$.  To
simplify subsequent notation we use the convention that the valuation
of an element $\alpha\in K$ with respect to a real place $v$ is given
by $\val_v(\alpha)=1$ if $\alpha$ is negative with respect to $v$ and
$0$ otherwise.
  
Let $b'\in\OO$ be a prime relatively prime to $\mu$, congruent to
$b\bmod{a}$, and such that $\val_v(b')=1$.  Such a $b'$ exists by
Dirichlet's theorem.  This is clear in the archimedean case. In the
nonarchimedean case, this can be done by finding an ideal
$\fb\subset\OO_K$ in the same ideal class as $I_{v,\OO_K}^{-1}$ with $\fb
I_{v,\OO_K}$ having a generator $b'\equiv b\bmod{(a\OO)}$.  Note that
$(a,b)\overset{1,u}{\Longrightarrow}(a,b')$.

For a prime $w$ of $K$, denote by $\Jacobi{\ast,\ast}w_\mu$ the power
residue symbol of degree $\mu$ at $w$ (cf. \cite[p. 85]{bms}).  Find a
prime $a'$ of $\OO$ congruent to $a\bmod{b'}$ such that
\begin{enumerate}[\upshape (a)]
\item
\label{rest}
$\Jacobi{a', b'}{v_i}_\mu=1$  for all places $v_i$ in
$S$ or dividing $\mu$ except $v$ and 
\item
\label{well}
$\Jacobi{a',
  b'}{v}_\mu=\Jacobi{a', b'}{b'\OO}_\mu^{-1}$.
\end{enumerate}
Note that
\begin{enumerate}
\item
\label{rated}
In the nonarchimedean case, \eqref{rest} 
are all congruence conditions modulo
sufficiently high powers of $I_{v_i,\OO_K}$.
\item
\label{poorly}
\eqref{well} is a condition modulo a power of the product of the
ideals above $v$ and $b'$.  That it is nonempty follows from the fact
that the map given by $\Jacobi{\ast, b'}{v}_\mu$ is surjective if
$\val_v(b')\equiv1\pmod{\mu}$.
\end{enumerate}
We can see that such an $a'$ exists
from Dirichlet's theorem.  Note that $a'$ and $b'$ are relatively
prime and $(a,b')\overset{1,\ell}{\Longrightarrow}(a',b')$.

Observe that by the reciprocity law $\Jacobi{a', b'}{a'\OO}_\mu=1$.
This implies that  $b'\equiv x^\mu\bmod(a'\OO)$ for some residue $x$ using
\cite[(A.16)]{bms}; cf. \cite[p.~18]{mrs}.  By the
generalization of Dirichlet's theorem to $\mathbb{Q}$-split primes,
see \cite[Theorem 3.3]{mrs} there are infinitely many odd,
degree-$1$ principal prime ideals $\fq$ with a generator
\mbox{$\text{\koppa\char19}\equiv x\bmod{(a'\OO})$}.  Then for all
these {\koppa\char19} and for all $m\equiv1 \bmod{\phi(a'\OO)}$ we
have $(a',b')\overset{1,u}{\Longrightarrow}(a',
\mbox{{\koppa\char19}}^{\mu m})$.  Hence we are done.
\end{proof}

\begin{proof}[Proof of Theorem~{\rm \ref{pinky}}]
Suppose $S$ contains a finite place or $\#\OO^{\times}=\infty$
and $K$ has at least one real embedding.  Let 
$A=\left[\begin{smallmatrix} a & b\\c &  d\end{smallmatrix}\right]\in
\SLT(\OO)$.
Proceed as in the proof of \cite[Section 4]{mrs} only use Lemma
\ref{lemma} instead of \cite[Lemma~4.4]{mrs}.  Thus we don't need to
use \cite[Lemma~4.3]{mrs} and we end up 
showing $(a,b)\stackrel{7}{\Longrightarrow}(1,0)$ instead of 
$(a,b)\stackrel{8}{\Longrightarrow}(1,0)$ as in \cite[Eq. (21)]{mrs}.
Hence $A$ is the product of $8$ elementary matrices beginning with 
a lower triangular matrix.
\end{proof}

\section{Theorem \texorpdfstring{\ref{pink}}{1.3}}
\label{founding}

\subsection{Division Chains}
\label{ready}

\begin{definition1}{\rm (cf.~\protect{\cite[Section 2]{cw}}.)}
\label{squid}
{\rm
Let $(a,b)\in\RR(\OO)$ as in \eqref{shrimp}.  A
\emph{division chain}  of length $k$ starting with $(a,b)$ is a sequence of
equations 
\begin{eqnarray}
\label{rose}  a &=& q_1b + r_1\\
 \nonumber b &=& q_2r_1 + r_2\\
\nonumber  & \mathrel{\makebox[\widthof{=}]{\vdots}}& \\
 \nonumber r_{k-3} &=& q_{k-1}r_{k-2}+r_{k-1}\\
\nonumber  r_{k-2} &=& q_kr_{k-1}+r_k
\end{eqnarray}
\noindent with $q_i\in\OO$, $1\leq i\leq k$.
The division chain is \emph{terminating} if $r_k=0$.
Notice that since $a$ and $b$ are relatively prime, in the
terminating case $r_{k-1}$ must be a unit.
}
\end{definition1}
\begin{remark1}
\label{tuna}
{\rm
The division chains of 
Definition \ref{squid} are closely 
related to the row reductions of Definition
\ref{dime}.  The division chain in \eqref{rose} of length $k$
starting with $(a,b)\in \RR(\OO)$ is equivalent to
\[
(a,b)\overset{k,\ell}{\Longrightarrow}\begin{cases}
(r_{k-1},r_{k})\text{ if $k$ is even}\\
(r_{k}, r_{k-1})\text{ if $k$ is odd}.
\end{cases}
\]
}
\end{remark1}

 The following lemma is elementary:

\begin{lemma}
  \label{extra}
 We have
 $b\equiv v\bmod a$  for  $v\in\OO^{\times}$
if and only if there
 exists a terminating division chain of length $2$ starting with $(b,a)$.
\end{lemma}

\subsection{Terminating division chains of length \texorpdfstring{\protect{\boldmath{$2$}}}{2}}
\label{copper}

Consider the matrix
\begin{equation}
  \label{weather}
A=   \begin{bmatrix} a & b \\ c& d\end{bmatrix}\in\SLT(\OO).
\end{equation}

Assume in this subsection that there is a terminating division chain of
length $2$ starting with $(b,a)$.  Therefore, by Lemma \ref{extra}, 
we have $b\equiv v \bmod a$, or $b-v=ax$ for $x\in\OO$,
with a unit $v\in\OO^{\times}$.

\begin{proposition}
  \label{granted}
  \begin{equation*}
    AU(-x)L(v^{-1}(1-a))U(-v)=L(w)
  \end{equation*}
  for some $w\in\OO$.
\end{proposition}
\begin{proof}
  Multiplying matrices verifies that
  \begin{equation*}
    AU(-x)L(v^{-1}(1-a))U(-v)=:B =\begin{bmatrix}
    1 & 0\\\ast & \ast\end{bmatrix}.
  \end{equation*}
  But the entry $B_{22}$ must be $1$ since $B\in\SLT(\OO)$.
  Hence $B=L(w)$ for some $w\in\OO$.
\end{proof}
\begin{theorem}
  \label{ghost}
  Let $A$ be as in \eqref{weather} assume there is a terminating
division chain of length $2$ starting with $(b,a)$.  
Then $A$ can be written as product of at most $4$
  elementary matrices with the first one lower triangular.
\end{theorem}
\begin{proof}
  From Proposition \ref{granted} we have
  \begin{equation}
    \label{tile}
    A=L(w)U(-v)^{-1}L(v^{-1}(1-a))^{-1}U(-x)^{-1} .
  \end{equation}
  But for any $s\in \OO$ we have $U(s)^{-1}=U(-s)$ and
  $L(s)^{-1}=L(-s)$.  Hence \eqref{tile}
  becomes
  \begin{equation*}
    A=L(w)U(v)L(v^{-1}(a-1))U(x) .
    \end{equation*}
      \end{proof}

\subsection{General Matrices in \texorpdfstring{\protect{\boldmath{$\SLT(\OO)$}}}{SL2(O)}}

\begin{theorem}
  \label{chain}
Let $A$ be as in \eqref{weather}.  If there exists a terminating
division chain of length $k>1$ starting at
\[
\begin{cases}
  (a,b)\quad\mbox{if}\quad\mbox{$k$ is odd}\\
  (b,a)\quad\mbox{if}\quad\mbox{$k$ is even ,}
\end{cases}
\]
then $A$ can be written as the product of at most $k+2$ elementary matrices
with the first one lower triangular.
\end{theorem}
\begin{proof}
We proceed by induction on $k$.  The $k=2$ case is Theorem
\ref{ghost}.

Suppose $k$ is odd.  Then by the definition of a terminating division chain
there exists $y\in\OO$ such that
\[
a-r=by
\]
and $(b,r)$ has a terminating division chain of length
$k-1$.  Then
\[
AL(-y)=\begin{bmatrix}
r&b\\
*&*
\end{bmatrix}
\]
is the product of $k+1$ elementary matrices with the first one
lower triangular by the induction
hypothesis.

The $k$ even case is handled similarly only switch the roles of $a$
and $b$ as well as multiply by $U(-y)$ instead of $L(-y)$.
\end{proof}

Note that this construction is similar to that used in
\cite[Corollary 2.3]{cw} except ours is more efficient, so we end up
with $k+2$ rather than the $k+4$ elementary matrices
produced by the construction of \cite[p.~496--498]{cw}. This accounts for
why our numbers are two smaller than theirs.

\subsection{The Generalized Riemann Hypothesis and
 the Proof of Theorem \texorpdfstring{\ref{pink}}{1.3}}
\label{yellow}

 The relevant Riemann hypothesis is most clearly stated in \cite[Theorem
  3.1]{len}.

\begin{riemannhypothesis}
\label{GRH}
The $\zeta$-function of $K(\zeta_n,\sqrt[n]{\OO^\times})$ satisfies
the Riemann hypothesis for all integers $n>0$.
\end{riemannhypothesis}

\begin{proof}[Proof of Theorem~{\rm \ref{pink}}]
Let $\OO$ be the $S$-integers in $K$ and $(a,b)\in\RR(\OO)$
as in \eqref{shrimp}.
Assume Hypothesis \ref{GRH}.
Then by \cite[Theorem 2.2]{cw}  there is a terminating
division chain of length $5$ starting with $(a,b)$.
If $S$ contains at least one finite prime, then there
there is a terminating division chain of length $4$ starting with
$(a,b)$ by  \cite[Theorem 2.9]{cw}, attributed to
Lenstra.
If $K$ has a real place,
then \cite[Theorem 2.14]{cw} shows that there is a terminating division
chain of length $3$ starting with $(a,b)$.
Now apply Theorem \ref{chain}.
\end{proof}

Morgan, Rapinchuk, and Sury~\cite[Proposition~5.1]{mrs} show that if
$p>7$ is a prime, then not every matrix in $\SL_2(\mathbb{Z}[1/p])$ is a
product of $4$ elementary matrices.  Hence the bound of $5$ elementary
matrices if $K$ has a real embedding in Theorem~\ref{pink} assuming
Hypothesis~\ref{GRH} would be strict.

\section*{Acknowledgments}
We sincerely thank the referee for corrections and improvements
to the paper.

\bibliographystyle{plain}
\bibliography{jordan-zaytman2}

\begin{bibdiv}
\begin{biblist}

\bib{bms}{article}{
      author={Bass, H.},
      author={Milnor, J.},
      author={Serre, J.-P.},
       title={Solution of the congruence subgroup problem for {${\rm
  SL}_{n}\,(n\geq 3)$} and {${\rm Sp}_{2n}\,(n\geq 2)$}},
        date={1967},
        ISSN={0073-8301},
     journal={Inst. Hautes \'{E}tudes Sci. Publ. Math.},
      number={33},
       pages={59\ndash 137},
         url={http://www.numdam.org/item?id=PMIHES_1967__33__59_0},
      review={\MR{244257}},
}

\bib{cw}{article}{
      author={Cooke, George},
      author={Weinberger, Peter~J.},
       title={On the construction of division chains in algebraic number rings,
  with applications to {${\rm SL}_{2}$}},
        date={1975},
        ISSN={0092-7872},
     journal={Comm. Algebra},
      volume={3},
       pages={481\ndash 524},
      review={\MR{0387251}},
}

\bib{jz}{article}{
      author={Jordan, Bruce~W.},
      author={Zaytman, Yevgeny},
       title={Integral points on varieties defined by matrix factorization into
  elementary matrices},
        date={2019},
        note={arXiv:1901.09433},
}

\bib{len}{article}{
      author={Lenstra, H.~W., Jr.},
       title={On {A}rtin's conjecture and {E}uclid's algorithm in global
  fields},
        date={1977},
        ISSN={0020-9910},
     journal={Invent. Math.},
      volume={42},
       pages={201\ndash 224},
      review={\MR{0480413}},
}

\bib{mrs}{article}{
      author={Morgan, Aleksander~V.},
      author={Rapinchuk, Andrei~S.},
      author={Sury, Balasubramanian},
       title={Bounded generation of {${\rm SL}_2$} over rings of {$S$}-integers
  with infinitely many units},
        date={2018},
        ISSN={1937-0652},
     journal={Algebra Number Theory},
      volume={12},
      number={8},
       pages={1949\ndash 1974},
         url={http://dx.doi.org/10.2140/ant.2018.12.1949},
      review={\MR{3892969}},
}

\bib{vas}{article}{
      author={Vaser{\v{s}}te{\u\i}n, L.~N.},
       title={The group {${\rm SL}_{2}$} over {D}edekind rings of arithmetic
  type},
        date={1972},
     journal={Mat. Sb. (N.S.)},
      volume={89(131)},
       pages={313\ndash 322, 351},
      review={\MR{0435293}},
}

\bib{v}{article}{
      author={Vsemirnov, Maxim},
       title={Short unitriangular factorizations of {${\rm SL}_2(\Bbb
  Z[1/p])$}},
        date={2014},
        ISSN={0033-5606},
     journal={Q. J. Math.},
      volume={65},
      number={1},
       pages={279\ndash 290},
         url={http://dx.doi.org/10.1093/qmath/has044},
      review={\MR{3179662}},
}

\end{biblist}
\end{bibdiv}

\end{document}